\documentclass[review]{elsarticle}

\usepackage{lineno,hyperref}
\modulolinenumbers[5]

\journal{Journal of Number Theory}

\usepackage{amssymb,amsthm,amsmath,amsfonts,enumerate,stmaryrd}
\usepackage{fancyhdr}
\usepackage{mathtools} 
\usepackage{tikz-cd}
\usetikzlibrary{matrix}
\usepackage{dsfont}
\usepackage[makeroom]{cancel}
\usepackage{mathrsfs}
\usepackage{graphicx} 
\graphicspath{{figures/}} 

\usepackage{xcolor} 

\usepackage{amsmath} 
\usepackage{amssymb} 

\usepackage{booktabs} 
\usepackage{enumitem} 

\newenvironment{Sproof}{%
\proof}{\endproof}
\newtheorem{theorem}{Theorem}[section]

\theoremstyle{definition}

\newtheorem{proposition}[theorem]{Proposition}

\theoremstyle{remark}

\numberwithin{equation}{section}

\newcommand{\G}{\Gamma}

\newcommand{\C}{\mathbb{C}}

\newcommand{\Slz}{\text{SL}_2(\mathbb{Z})}

\begin{document}

\begin{frontmatter}

\title{Zeros of certain combinations of Eisenstein series of weight $2k, 3k,$ and $k+l$}

\author{Jetjaroen Klangwang}
\address{Department of Mathematics, Oregon State University, Corvallis, Oregon 97331}

\ead[url]{http://people.oregonstate.edu/$\sim$klangwaj/}

\ead{klangwaj@oregonstate.edu}


\begin{abstract}
We locate the zeros of the modular forms $E_k^2(\tau)  + E_{2k}(\tau), E_k^3(\tau) + E_{3k} (\tau),$ and $E_k(\tau)E_l(\tau) +E_{k+l}(\tau),$ where $E_k(\tau)$ is the Eisenstein series for the full modular group $\Slz$. By utilizing work of F.K.C. Rankin and Swinnerton-Dyer, we prove that for sufficiently large $k,l$, all zeros in the standard fundamental domain are located on the lower boundary $\mathcal{A} = \{ e^{i\theta} : \pi/2 \leq \theta \leq 2\pi/3\}$.
\end{abstract}


\end{frontmatter}

\section{Introduction and Statement of results}

Let $\G = \Slz$. The full modular group $\G$ acts on the upper half plane $\mathbb{H} = \{ \tau \in \C : \text{Im}(\tau)> 0\}$ by fractional linear transformations. A standard fundamental domain of this action is given by 
	$$
	\mathcal{F}=\left\{ |\tau| >1 \; \text{and}\; 0< \text{Re}(\tau) < \frac{1}{2}\right\} \cup \left\{ |\tau| \geq 1 \; \text{and} \; -\frac{1}{2} \leq \text{Re}(\tau) \leq 0\right\}.
	$$

Let $k$ be an even integer. For $k \geq 2$, the classical (normalized) Eisenstein series of weight $k$ for $\G$ is defined by 
\begin{equation}\label{def:Eisensteinseries}
E_k(\tau) = \dfrac{1}{2} \sum_{\substack{c,d \in \mathbb{Z} \\ \gcd (c,d) =1}} (c\tau + d)^{-k}.
\end{equation}

The location of the zeros of Eisenstein series has been studied since  1960s. 
Wohlfahrt  \cite{wohlfahrt1963nullstellen} showed in 1963 that for $4 \leq k \leq 26$, all zeros in the fundamental domain $\mathcal{F}$ of $E_k(\tau)$ lie on the unit circle $|\tau|=1$ and conjectured that this holds for $k \geq 4$. The range of $k$ was extended to $4 \leq k \leq 34$, and $k=38$ by R.A. Rankin in \cite{rankin1969zeros}. Eventually, Wohlfahrt's conjecture was proved by R.A. Rankin's daughter, F. K. C. Rankin, together with Swinnerton-Dyer in their famous
paper \cite{rankin1970zeros}. 
		
The argument of F. K. C. Rankin and Swinnerton-Dyer has been generalized to Eisenstein series for different groups \cite{miezaki2007zeros,shigezumi2007zeros,garthwaitelevel22010zeros, hahn2007zeros},  other modular forms \cite{getz2004generalization, gun2006zeros}, and certain weakly holomorphic modular forms \cite{duke2008zeros, garthwaite2012zeros, haddock2014zeros}. 

Recently, Reitzes, Vulakh and Young \cite{reitzes2017zeros} showed that for $k \geq l \geq 14$, all zeros in the fundamental domain $\mathcal{F}$ of the cusp form $E_k(\tau) E_l(\tau) - E_{k+l}(\tau)$ are either located on the lower boundary or on the left side boundary $\{ \tau \in \mathcal{F} : \text{Re}(\tau) = -1/2\}$ of the standard fundamental domain. 

The aim of this paper is to generalize the approach of F.K.C. Rankin and Swinnerton-Dyer in \cite{rankin1970zeros} as well as techniques Reitzes et al. used in \cite{reitzes2017zeros} to show that for $n = 2, 3$, all zeros in the fundamental domain $\mathcal{F}$ of the modular forms of weight $nk$ defined by 
$$
E_k^n(z) + E_{nk}(z),
$$
and all zeros in the fundamental domain $\mathcal{F}$ of the modular form of weight $k +l$ defined by 
$$
E_k(\tau)E_l(\tau) + E_{k+l}(\tau)
$$
lie on the lower bound boundary. Let us now state our results. 

\begin{theorem}\label{theorem:2k3k}
Let $k$ be even. All zeros of $E_k^2(\tau) + E_{2k}(\tau)$ for $k \geq 10$ and all zeros of $E_k^3(\tau) + E_{3k}(\tau)$ for $k \geq 16$ in the fundamental domain $\mathcal{F}$ are located on the arc $\mathcal{A} = \{e^{i\theta} : \pi/2 \leq \theta \leq 2\pi/3\}$. 
\end{theorem}

\begin{theorem}\label{theorem:k+l} If $k > l \geq 10$, then all zeros of $E_k(\tau)E_l(\tau)  + E_{k+l}(\tau)$ in the fundamental domain $\mathcal{F}$ are located on the arc $\mathcal{A} = \{e^{i\theta} : \pi/2 \leq \theta \leq 2\pi/3\}$. 
\end{theorem}


\section{Work of F.K.C. Rankin and Swinnerton-Dyer}
	In this section, we brieftly discuss the argument of F.K.C. Rankin and Swinnerton-Dyer on the zeros of the Eisenstein series $E_k(\tau)$ for the modular group $\G$. In \cite{rankin1970zeros}, F.K.C. Rankin and Swinnerton-Dyer use the elementary tools from calculus such as approximations of trigonometric functions, the intermediate value theorem, and the valence formula from the theory of modular forms to prove the following theorem. 
	 
	\begin{theorem} \label{rsdtheorem} \cite{rankin1970zeros}
	For even $k \geq 4$, all zeros of $E_k(\tau)$ in the fundamental domain $\mathcal{F}$ are located on the arc 
		$$
		\mathcal{A} = \left\{ e^{i \theta} : \frac{\pi}{2} \leq \theta \leq \frac{2\pi}{3} \right\}.
		$$
	\end{theorem}
	
\begin{Sproof} In 1960s, Wohlfahrt and R.A. Rankin gave partial results of the zeros of the Eisenstein series for $\Slz$ in  \cite{wohlfahrt1963nullstellen} and \cite{rankin1969zeros} for even $4 \leq k \leq 34$, and $k=38$. To prove Theorem \ref{rsdtheorem}, F.K.C. Rankin and Swinnerton-Dyer consider even $k \geq 12$ and write
$$
k = 12n + s
$$
with uniquely determined $n \in \mathbb{Z}$ and $s \in \{0, 4,6,8,10,14\}$. 

Note that any nonzero modular form $f(\tau)$ of weight $k$ for $\G$ satisfies the valence formula
	\begin{equation}\label{valenceformula}
	\nu_\infty (f) + \frac{1}{2} \nu_i(f) + \frac{1}{3} \nu_\rho(f) + \sum_{\tau \in \mathcal{F} \backslash \{i, \rho\}} \nu_\tau(f) = \frac{k}{12},
	\end{equation}
	where $\rho = e^{2\pi i / 3}$ and $\nu_\tau(f)$ is the order of vanishing of $f$ at $\tau \in \mathcal{F}$. With the above notation $k = 12 n + s$, we have that $E_k(\tau)$ satisfies
$$
\frac{1}{2} \nu_i(E_k) + \frac{1}{3} \nu_\rho(E_k) + \sum_{\tau \in \mathcal{F} \backslash \{i, \rho\}} \nu_\tau(E_k) = n+\frac{s}{12},
$$
where $\nu_\infty (E_k) = 0$ since $E_k(\tau)$ is holomorphic at $\infty$ and the constant term in its $q$-expansion equals 1. Also, by considering all possible values of $s \in \{0,4,6,8,10,14\}$, we find that $s/12$ determines the order of zeros at $\tau = i, \rho$. 

Then to show that all zeros of $E_k(\tau)$ are located on the lower arc $\mathcal{A}$, it suffices to show that a function $E_k(e^{i\theta})$ has at least $n$ zeros on $(\pi/2, 2\pi/3)$. 

F.K.C. Rankin and Swinnerton-Dyer consider the function 
	\begin{equation}\label{def:Fk}
	F_k(\theta):=e^{ik\theta/2}E_k(e^{i\theta}),
	\end{equation} 
	 which clearly share the same set of zeros with the function $E_k(e^{i\theta})$ on $[\pi/2, 2\pi/3]$. Moreover, $F_k(\theta)$ is real on $[\pi/2, 2\pi/3]$ by Proposition 2.1 of \cite{getz2004generalization}.
	 
	 By the definition of $E_k(\tau)$ given in (\ref{def:Eisensteinseries}), we can write
	$$
	F_k(\theta) = \dfrac{1}{2} \sum_{\substack{c,d \in \mathbb{Z} \\ \gcd (c,d) =1}} (ce^{i\theta/2} + de^{-i\theta/2} )^{-k}.
	$$
 Let $M_k(\theta)$ denote the sum in the series with $c^2 + d^2 = 1$, $N_k(\theta)$ denote the sum in the series with $(c, d) = \pm (1, 1)$ and denote the remainder of the series $R_k(\theta)$. Then 
	\begin{equation}\label{def:maintermofFk}
	F_k(\theta)  = M_k(\theta) + N_k(\theta) + R_k(\theta),
	\end{equation}
	where 
	\begin{equation}\label{def:MkNk}
	M_k(\theta) := 2\cos \left(\frac{k \theta}{2} \right), \quad \text{and} \quad  N_k(\theta) := \left( 2\cos \left( \frac{\theta}{2}\right) \right)^{-k}.
	\end{equation}
	and
	$$
	R_k(\theta) = \left( 2i \sin\left( \frac{\theta}{2}\right) \right)^{-k} + \dfrac{1}{2} \sum_{\substack{\gcd(c,d) =1 \\ c^2 + d^2 \geq 5}} \frac{1}{(ce^{i\theta/2} + d e^{-i\theta/2})^k}.
	$$
By the triangle inequality, approximation on trigonometric functions, and the integral test, they prove that for $k \geq 12$, $|R_k(\theta)|$ is monotonically decreasing as a function in $k$ and bounded above by
	\begin{equation}\label{def:remaindertermofFk}
	|R_k| \leq \left(\frac{1}{2} \right)^{k/2} + 4 \left( \frac{2}{5} \right)^{k/2} + \frac{20\sqrt{2}}{k-3} \left( \frac{2}{9} \right)^{(k-3)/2} \leq 0.3563. 
	\end{equation}
Hence, (\ref{def:remaindertermofFk}), and the fact that $|N_k(\theta)| = |(2\cos(\theta/2))^{-k}| \leq 1$ on $[\pi/2, 2\pi/3]$,
\begin{equation}\label{smallerthan2}
\left| F_k(\theta) - M_k(\theta) \right|  \leq |N_k(\theta)| + |R_k(\theta)| \leq 1.03563.
\end{equation}
By taking $\theta_m := 2m\pi/k$ where $m$ ranges over integers so that $\theta_m \in [\pi/2, 2\pi/3]$, $M_k(\theta_m) = 2\cos(m\pi) = 2(-1)^m$ and therefore the lower bound given in (\ref{smallerthan2}) tells us that $F_k(\theta_m)$ has different sign for consecutive integers $m$'s. 

We now apply the intermediate value theorem to conclude that the minimum number of zeros of the function $F_k(\theta)$ and hence $E_k(e^{i\theta})$ in $(\pi/2, 2\pi/3)$ is the number of integers in $[k/4, k/3]$ minus 1. Using the parameterization $k = 12n + s$ where $s \in \{0, 4, 6, 8, 10, 14\}$, the number of integers in $[k/4, k/3]$ equals the number of integers in $[3n+s/4, 4n + s/3]$. We see that for each choice of $s$ there are $n + 1$ integers in this interval.  Thus, the function $E_k(e^{i\theta})$ has at least $n$ zeros in $(\pi/2, 2\pi/3)$ and this completes the proof of Theorem \ref{rsdtheorem}.
\end{Sproof}

\section{Locating the zeros of $E_k^n(\tau) + E_{nk}(\tau)$}	

For even $k\geq 4$ and $n =2, 3$, we write 
$$
k = \left(\frac{12}{n}\right) l_n + s_n 
$$
where $l_n \in \mathbb{Z}$ and $s_n \in \{ 0, 2, \dots,\left( 12/n\right) -2 \}$. The valence formula (\ref{valenceformula}) guarantees that the modular form $E_k^n(\tau) + E_{nk}(\tau)$ has zeros of order at least $n s_n/12$  at $z  =i, \rho$ and has $l_n$ zeros  in $\mathcal{F} \backslash \{i, \rho \}$ (counting multiplicities).  

This argument and Proposition 2.1 of \cite{getz2004generalization} imply that to prove that all zeros of $E_k^n(\tau) + E_{nk}(\tau)$ lie on the arc $\mathcal{A}$, it suffices to prove that the real-valued function 
	\begin{equation}\label{def:Fnk}
	F_{n,k}(\theta) := e^{ink\theta/2} \left(E_k^n+ E_{nk} \right) (e^{i\theta}) 
	\end{equation}
	has at least $l_n$ zeros in the open interval $(\pi/2, 2\pi/3)$. 

\subsection{Extraction of the main and error terms} 
Similar to the method of F.K.C. Rankin and Swinnerton-Dyer reviewed in Section 2, we begin with writing the function $F_{n,k}(\theta)$ as a sum of main and remainder terms and then give an upper bound of the remainder term. 

\begin{proposition} \label{prop:Fnkmainplusremainderterms}  For even $k \geq 4$, for $n =2,3$ and for $\theta\in [\pi/2, 2\pi/3]$, we have 
		\begin{equation}
		F_{n,k}(\theta) = M_{n,k}(\theta) + R_{n,k}(\theta)
		\end{equation}
		where 
		\begin{equation}\label{def:Fnkmainterm}
		M_{n,k}(\theta) =  \left( M_k(\theta) + N_k(\theta) \right)^n+ M_{nk}(\theta) + N_{nk}(\theta)
		\end{equation}
		with $M_k(\theta)$ and $N_k(\theta)$ are defined in (\ref{def:MkNk}) and 
		\begin{equation}\label{def:Fnkupperboundremainderterm}
		|R_{n,k}(\theta)| \leq \left\{ \begin{array}{rl}
0.56875 &\mbox{ if $n=2$ and $k \geq 10$,} \\
0.17999 &\mbox{ if $n = 3$ and $k \geq 16$.}
\end{array} \right.
		\end{equation}
\end{proposition}

\begin{proof} We can write 
	$$
	F_{n,k}(\theta) = (F_k^n + F_{nk})(\theta). 
	$$
where $F_k(\theta)$ is defined in (\ref{def:Fk}).
Expanding the right hand side using (\ref{def:maintermofFk}),
we derive
		\begin{equation} \label{proof:simplifiedFkn}
		\begin{aligned}
		F_{n,k}(\theta) 
		&= \left( M_k(\theta) + N_k(\theta) \right)^n+  M_{nk}(\theta) + N_{nk}(\theta) \\
		&\quad + \sum_{i=1}^n \binom{n}{i}   \left( M_k(\theta) + N_k(\theta) \right)^i R_k^{n-i}(\theta) + R_k^n(\theta) + R_{nk}(\theta).
		\end{aligned}
		\end{equation}
Let $M_{n,k}(\theta)$ and $R_{n,k}(\theta)$ be the main and remainder terms of $F_{n,k}(\theta)$ obtained from the first and second line of (\ref{proof:simplifiedFkn}) respectively.  Since $|M_k(\theta)| = |2\cos(k\theta/2)| \leq 2$ and $|N_k(\theta)| = |(2\cos(\theta/2))^{-k}| \leq 1$ on $[\pi/2, 2\pi/3]$, the triangle inequality gives us
			$$
			|R_{n,k}(\theta)| \leq \sum_{i=1}^n \binom{n}{i} 3^i |R_k(\theta)|^{n-i} + |R_k(\theta)|^n + |R_{nk}(\theta)|. 
			$$
Recall that $|R_k(\theta)|$ is monotonically decreasing as a function in $k$ so the  term $|R_{n,k}(\theta)|$ is also. Evaluating the upper bound of $|R_k(\theta)|$ in (\ref{def:remaindertermofFk}) at $k = 10$ (and $k = 16$), we easily obtain the upper bound for $|R_{n,k}(\theta)|$ in  (\ref{def:Fnkupperboundremainderterm}). 
\end{proof}

\subsection{Sample points} 
	Let $k \geq 10$ be an even integer and let $n \in \{2, 3\}$. We define 
	$$
	\theta_{nk}(m) := \frac{2m\pi}{nk}
	$$
	where $m$ ranges over integers so that $\theta_{nk}(m) \in [\pi/2, 2\pi/3]$. Observe that 
	$$
	\theta_{nk}(m)\in \left[ \frac{\pi}{2}, \frac{2\pi}{3} \right] \Leftrightarrow m \in \left[ \frac{nk}{4}, \frac{nk}{3} \right].
	$$
	
	Our goal for the rest of this section is to show that the function $F_{n,k}(\theta_m)$ is strictly positive or negative according to the parity of $m \in [ nk/4, nk/3]$. Since $F_{n,k}(\theta) = M_{n,k}(\theta) + R_{n,k}(\theta)$ by Proposition \ref{prop:Fnkmainplusremainderterms}, we show that for all integers $m \in [nk/4, nk/3]$, a lower bound of $(-1)^m M_{n,k}(\theta_{nk}(m))$ is greater than the upper bound of $|R_{n,k}(\theta)|$ given in Proposition \ref{prop:Fnkmainplusremainderterms}. 

\subsection{Bounding the main term} We first give a lower bound on $(-1)^m M_{2,k}(\theta_{2k}(m))$.

	\begin{proposition} \label{prop:lowerboundmaintern=2} For even $k \geq 10$ and $\theta_{2k}(m) \in [\pi/2, 2\pi/3]$, we have 
		$$
		(-1)^m M_{2,k}(\theta_{2k}(m)) \geq 1.64849. 
		$$
	\end{proposition} 
	
	\begin{proof} We observe that
	$$
	\theta_{2k}(m)= \frac{m\pi}{k} \in \left[\frac{\pi}{2}, \frac{2\pi}{3} \right] \Leftrightarrow m \in \mathbb{Z} \cap \left[ \frac{k}{2}, \frac{2k}{3} \right].
	$$
	Substituting $\theta_{2k}(m)$ into  (\ref{def:Fnkmainterm}), we obtain
			\begin{equation}\label{def:plugintheta2kminM2k}
			\begin{aligned}
			M_{2,k}(\theta_{2k}(m)) &= \left( M_k(\theta_{2k}(m)) + N_k(\theta_{2k}(m)) \right)^2 +  2(-1)^m +  N_{2k}(\theta_{2k}(m)).
			\end{aligned}
			\end{equation}
We note that for even $k$, and for $n =2, 3$, it is straightforward to check that the derivative of 
\begin{equation}\label{lemma1}
N_{nk}(\theta_{nk}(m)) = \left( 2\cos\left( \frac{m\pi}{nk} \right)\right)^{-nk}
\end{equation} 
is positive for $m \in [nk/4, nk/3]$ and therefore $N_{nk}(\theta_{nk}(m))$ is positive and monotonically decreasing as a function of $m$ in that interval. From this and (\ref{def:plugintheta2kminM2k}), 
	\begin{equation}\label{ineq:maintermwhenmisoddmstar}
		(-1)^m M_{2,k}(\theta_{2k}(m)) \geq 2 - 2 \left( 2\cos\left(\frac{m_{odd} \pi}{2k} \right)\right)^{-2k}. 
	\end{equation}
where $m_{odd}$ is the largest odd integer in $[k/2, 2k/3]$. Considering $k \pmod 6$, 
			\begin{equation}\label{def:mstar}
			m_{odd}= \frac{2k}{3} - \frac{3-r}{3},
		 	\end{equation}
where $k \equiv r \pmod 6$ with $r \in \{ 0,\pm 2 \}$. 
Substituting  (\ref{def:mstar}) into (\ref{ineq:maintermwhenmisoddmstar}), we obtain
	$$
	(-1)^m M_{2k}(\theta_{2k}(m)) \geq 2 - 2 \left( 2\cos\left( \frac{\pi}{3} - \left(  \frac{3-r}{3}\right) \frac{\pi}{2k}\right) \right)^{-2k} . 
	$$
By Lemma 2.2 of \cite{reitzes2017zeros} and the identity $\cos(\theta) = \sin(\pi/2 - \theta)$, the right hand side is monotonically increasing as a function in $k$. 
Hence, for $k \equiv 0 \pmod 6$ and $k \geq 12$,
	\begin{equation}\label{ineq:2klowerbound1}
	\begin{aligned}
	 (-1)^m M_{2k}(\theta_{2k}(m)) \geq  2 - 2 \left( 2\cos\left( \frac{\pi}{3} - \frac{\pi}{2(12)}\right) \right)^{-2(12)} \geq 1.98223.
	\end{aligned}
	\end{equation}
Applying this argument to the cases $k \equiv -2, 2 \pmod 6$, we obtain that  for $k \equiv -2 \pmod 6$ and $k \geq 10$, 
\begin{equation}\label{ineq:2klowerbound2}
(-1)^m M_{2k}(\theta_{2k}(m)) \geq  2  2 \left( 2\cos\left( \frac{\pi}{3} - \frac{5\pi}{3(2(10))}\right) \right)^{-2(10)} \geq 1.99804,
\end{equation}
and for $k \equiv 2 \pmod 6$ and $k \geq 14$, 
\begin{equation}\label{ineq:2klowerbound3}
(-1)^m M_{2k}(\theta_{2k}(m)) \geq  2 - 2 \left( 2\cos\left( \frac{\pi}{3} - \frac{\pi}{3 (2(8))}\right) \right)^{-2(8)} \geq 1.64849.
\end{equation}
By (\ref{ineq:2klowerbound1}), (\ref{ineq:2klowerbound2}), and (\ref{ineq:2klowerbound3}), we have proved Proposition \ref{prop:lowerboundmaintern=2}.	
	\end{proof}

	Next, we give a lower bound of $(-1)^m M_{3,k}(\theta_{3k}(m))$. The proof is based on the concept of the proof of Proposition \ref{prop:lowerboundmaintern=2}. 
	
	\begin{proposition}
	\label{prop:lowerboundmaintern=3} For even $k \geq 16$ and $\theta_{3k}(m) \in [\pi/2, 2\pi/3]$, we have 
		$$
		(-1)^m M_{3,k}(\theta_{3k}(m)) \geq 0.32869. 
	    $$
	\end{proposition}
	
	\begin{proof} We observe that
	$$
	\theta_{3k}(m)= \frac{2m\pi}{3k} \in \left[\frac{\pi}{2}, \frac{2\pi}{3} \right] \Leftrightarrow m \in \mathbb{Z} \cap \left[ \frac{3k}{4}, k \right].
	$$
	Substituting $\theta_{3k}(m)$ into  (\ref{def:Fnkmainterm}), we obtain
$$
	\begin{aligned}
	M_{3,k}(\theta_{3k}(m)) &=  \left( M_k(\theta_{3k}(m)) + N_k(\theta_{3k}(m)) \right)^3 +  2(-1)^m  + N_{3k}(\theta_{3k}(m)).
	\end{aligned}
$$
Assuming $m \in  [3k/4, k]$ is even. Then $M_k(\theta_{3k}(m)) = 2 \cos(m\pi/3)=-1,2$ and 
$$
	\begin{aligned}
	M_{3,k}(\theta_{3k}(m)) &\geq\left(-1+ N_k(\theta_{3k}(m)) \right)^3 +2  +N_{3k}(\theta_{3k}(m)).
	\end{aligned}
$$
By (\ref{lemma1}), for $k \geq 16$ and for even $m \in [3k/4, k]$,  
	\begin{equation}\label{ineq:3klowerbound1}
	M_{3,k}(\theta_{3k}(m)) > (-1 + 2^{-k/2})^3 +2 + 2^{-3k/2} >1. 
	\end{equation}
Assume $m\in  [3k/4, k]$ is odd.  Since $M_k(\theta_{3k}(m)) = 2\cos(m\pi/3)=-2,$ or $1$, 
	\begin{equation}\label{boundmaintermn=3odd}
	\begin{aligned}
	M_{3,k}(\theta_{3k}(m)) \leq \left(1+N_k(\theta_{3k}(m))  \right)^3  -2  +N_{3k}(\theta_{3k}(m)).
	\end{aligned}
	\end{equation}
By (\ref{lemma1}), the right hand side is monotonically increasing as a function of odd number $m \in [3k/4, k]$. Plugging $m = k-1$ in (\ref{boundmaintermn=3odd}), we obtain 
	\begin{equation}\label{boundmaintermn=3oddplugink-1}
	\begin{aligned}
	M_{3,k}(\theta_{3k}(m))  &\leq \left(1+ \left( 2\cos \left( \frac{\pi}{3} - \frac{\pi}{3k} \right) \right)^{-k} \right)^3  -2  +\left( 2\cos \left( \frac{\pi}{3} - \frac{\pi}{3k}\right) \right)^{-3k}.
	\end{aligned}
	\end{equation}
By Lemma 2.2 of \cite{reitzes2017zeros} and the identity $\cos(\theta) = \sin(\pi/2 - \theta)$, the right hand side of (\ref{boundmaintermn=3oddplugink-1}) is monotonically decreasing as a function in $k$. Evaluating $k = 16$ in (\ref{boundmaintermn=3oddplugink-1}), we have that for $k \geq 16$ and for odd $m \in [3k/4, k]$,
\begin{equation} \label{ineq:3klowerbound2}
M_{3,k}(\theta_{3,k}(m)) \leq  -0.32869.
\end{equation}
	Therefore, by (\ref{ineq:3klowerbound1}), and (\ref{ineq:3klowerbound2}), the proof is completed. 
	\end{proof}
	
\subsection{Proof of Theorem \ref{theorem:2k3k}}  
\begin{proof} Recall that the function $\left( E_k^n + E_{nk} \right) (e^{i\theta})$ and the real-valued function $F_{n,k}(\theta) = e^{ink\theta/2} \left( E_k^n + E_{nk} \right) (e^{i\theta})$ have the same zero set on $[\pi/2, 2\pi/3]$ where $F_{n,k}(\theta)$ can be extracted as 
$$
	F_{n,k}(\theta) = M_{n,k}(\theta) + R_{n,k}(\theta),
$$
where Propositions \ref{prop:Fnkmainplusremainderterms}, \ref{prop:lowerboundmaintern=2} and \ref{prop:lowerboundmaintern=3}, showed that for $n = 2, 3$,
$$
	(-1)^m M_{n,k}(\theta_{nk}(m))> |R_{n,k}|
$$
	for large enough $k$ and for $\theta_{nk}(m) = 2m\pi/nk \in [\pi/2, 2\pi/3]$.

Thus, $F_{n,k}(\theta_{nk}(m))$ is strictly positive or negative according as $m$ is even or odd in $[nk/4, nk/3]$. 	
	 Then the intermediate value theorem guarantees that the minimum number of zeros of the function $F_{n,k}(\theta)$ and hence $E_k^n(e^{i\theta}) + E_{nk}(e^{i\theta})$ equals the number of $\theta_{nk}(m)$ in $[\pi/2, 2\pi/3]$ minus 1. 
	
Using the parametrization $k = (12/n)l_n + s_n$ where $s_n \in \{0, 2, \dots, s_n-2\}$,  the number of $\theta_{nk}(m)$ in $[\pi/2, 2\pi/3]$ equals the number of integers in $[nk/4, nk/3] = [3l_n + ns_n/4, 4l_n +ns_n/3]$. For $n =2$ or $3$, it can be shown easily that there are $l_n + 1$ integers in that interval. Hence, we conclude that $(E_k^n + E_{nk})(e^{i\theta})$ has at least $l_n$ zeros on $(\pi/2, 2\pi/3)$

As $E_k^n (\tau) + E_{nk}(\tau)$ can have at most $l_n$ nontrivial zeros in $\mathcal{F} \backslash \{ i, \rho\}$ as described at the beginning of Section 3 and the above argument shows that there are at least $l_n$ zeros on the arc $\mathcal{A}$, we finish the proof of Theorem \ref{theorem:2k3k}. 
\end{proof}


\subsection{Higher values of $n$}
	Computational evidence shows that the result in Theorem \ref{theorem:2k3k} does not extend to $n \geq 4$. When $n = 4, 5$ and $6$, the remainder term $R_{n,k}(\theta)$ is getting bigger than the main term $M_{n,k}(\theta_{nk}(m))$ as the values of $\theta$ get closer and closer $2\pi/3$. It would be very interesting to see what result holds for higher $n$. We leave this an open problem. 

\section{Locating the zeros of $E_k(\tau)E_l(\tau)+E_{k+l}(\tau)$}
Let $k > l \geq 10$ be even integers and consider
$$
E_k(\tau) E_l(\tau) + E_{k+l}(\tau)
$$
By symmetry, we assume that $k >l$ (the case $k=l$ is discussed in Section 3).  This modular form of weight $k +l$ is defined analogously to the cusp form
$$
E_k(\tau) E_l(\tau) - E_{k+l}(\tau).
$$
which appeared in the work of Reitzes et al. in \cite{reitzes2017zeros}. In their paper, they prove that if $k$ and $l$ are sufficiently large, then all zeros of $E_k(\tau) E_l(\tau) - E_{k+l}(\tau)$ lie on the arc $\mathcal{A} = \{e^{i\theta} : \pi/2 \leq \theta \leq 2\pi/3]$ or on the left side boundary $\{ \tau \in \mathcal{F} : \text{Re}(\tau) = -1/2\}$. 

In contrast to their result, we prove  that all zeros of $E_k(\tau) E_l(\tau) + E_{k+l}(\tau)$ are located on the arc $\mathcal{A}$.

We begin by writting  $k+l = 12n + s$ with $n \geq 1$ and $s \in \{0,4,6,8,10,14\}$ and considering the related function 
	\begin{equation}\label{def:Gkl}
	G_{k,l}(\theta) := e^{i(k+l)\theta/2} (E_kE_l + E_{k+l})(e^{i\theta}). 
	\end{equation} 
This function is real on $[\pi/2, 2\pi/3]$ by the Proposition 2.1 of \cite{getz2004generalization}. Also, the zeros of $G_{k,l}(\theta)$ on $[\pi2, 2\pi/3]$ corresponds bijectively to the zeros of $E_k(\tau) E_l(\tau) + E_{k+l}(\tau)$ on the arc $\mathcal{A}$. 

Similar to the method of F.K.C. Rankin and Swinnerton-Dyer reviewed in Section 2, we will show that $G_{k,l}(\theta)$ has at least $n$ zeros on $(\pi/2, 2\pi/3)$.


\subsection{Extraction of the main and error terms}

\begin{proposition}\label{prop:maintermplusremainderterm}  For   even $k > l \geq 10$ and for $\theta \in [\pi/2, 2\pi/3]$, we have 
		$$
		G_{k,l}(\theta) = M_{k,l}(\theta) + R_{k,l}(\theta)
		$$
		where $M_{k,l}(\theta)=(M_k(\theta)+N_k(\theta))(M_l(\theta) + N_l(\theta)) + M_{k+l}(\theta) + N_{k+l}(\theta)$ with $M_k(\theta)$ and $N_k(\theta)$ are defined in (\ref{def:MkNk})
		and $|R_{k,l}(\theta)| \leq 0.39018$. 
\end{proposition}

\begin{proof} By (\ref{def:Gkl}) and (\ref{def:Fk}), we can write 
\begin{equation}\label{def:Gkl=FkFl+Fk+l}
G_{k,l}(\theta) = \left( F_kF_l + F_{k+l} \right) (\theta).
\end{equation}
Plugging in (\ref{def:maintermofFk}) into (\ref{def:Gkl=FkFl+Fk+l}),  we obtain
\begin{equation} \label{eq: proof}
\begin{aligned}
G_{k,l}(\theta)
&=\left( M_k(\theta)+N_k(\theta) \right) \left( M_l(\theta)+N_l(\theta) \right)  + M_{k+l}(\theta) + N_{k+l}(\theta)\\
&\quad + \left( M_k(\theta)+N_k(\theta) \right) R_l(\theta) + \left( M_l(\theta)+N_l(\theta) \right) R_k(\theta)\\
&\quad + R_k(\theta) R_l(\theta) + R_{k+l}(\theta).
\end{aligned}
\end{equation}
Let $M_{k,l}(\theta)$ be a sum of all terms in the first line  in (\ref{eq: proof}) and let $R_{k,l}(\theta)$ be the sum of all remaining terms. To bound $|R_{k,l}(\theta)|$, the triangle inequality along with the fact that $|M_k(\theta)| = |2\cos(k\theta/2)| \leq 2$ and $|N_k(\theta)| = |(2\cos(\theta/2))^{-k}| \leq 1$ on $[\pi/2, 2\pi/3]$ yield
\begin{equation}\label{lowerboundRkl}
|R_{k,l}(\theta)| \leq  3|R_{l}(\theta)| + 3|R_{k}(\theta)|+ |R_{k}(\theta)||R_{l}(\theta)| + |R_{k+l}(\theta)|\\
\end{equation}
With the upper bound of $|R_k(\theta)|$ given in  (\ref{def:remaindertermofFk}),  it is easy to see that $|R_{k,l}(\theta)|$ is also monotonically decreasing in both $k,l$.  Evaluating the bound in (\ref{lowerboundRkl}) at $k = 12$ and $l = 10$, we get the upper bound for $|R_{k,l}(\theta)|$ in Proposition \ref{prop:maintermplusremainderterm}. 
This completes the proof.
\end{proof}


\subsection{Sample points}

Let $k > l \geq 10$ be even integers, and define 
$$
\theta_m := \theta_{k+l}(m) = \frac{2m\pi}{k+l}.
$$
where $m$ ranges over integers so that $\theta_m \in [\pi/2, 2\pi/3]$. We observe that 
$$
\theta_m \in \left[\frac{\pi}{2}, \frac{2\pi}{3} \right] \Leftrightarrow m\in \left[ \frac{k+l}{4}, \frac{k+l}{3}\right].
$$ 

With the definition of $M_k(\theta)$ given in (\ref{def:MkNk}),  $$M_{k+l}(\theta_m) = 2\cos \left( (k+l)\frac{\theta_m}{2} \right) = 2\cos(m\pi)= 2(-1)^m$$ and the sum and difference trigonometric identities give us
$$
M_k(\theta_m) = 2\cos\left( k \frac{\theta_m}{2} \right) = 2\cos \left( (k+l) {\theta_m}{2} - k \frac{\theta_m}{2} \right)= (-1)^m M_l(\theta_m). 
$$
Inserting these in the main term $M_{k,l}(\theta)$ in Proposition \ref{prop:maintermplusremainderterm}, we find that 
$$
\begin{aligned}
M_{k,l}(\theta_m) &= ((-1)^m M_l(\theta_m) + N_k(\theta_m)) (M_l(\theta_m) + N_l(\theta_m)) + 2(-1)^m + N_{k+l}(\theta_m)\\
&= 2(-1)^m + N_k(\theta_m) N_l(\theta_m) + N_{k+l}(\theta_m)   \\
&\quad + (-1)^mM_l^2 \left( \theta_m \right) + (-1)^m M_l \left(\theta_m \right)  \left\{  N_{l}\left( \theta_m \right) + (-1)^m  N_{k}\left(\theta_m \right) \right\}.
\end{aligned}
$$
Since $N_k(\theta) = (2\cos(\theta/2))^{-k}$ is given in (\ref{def:MkNk}), $N_k(\theta)N_l(\theta) = N_{k+l}(\theta)$ and hence we write 
$(-1)^m M_{k,l} (\theta_m) = P_{k,l}(\theta_m) + Q_{k,l}(\theta_m)$
where $P_k(\theta_m)$ and $Q_k(\theta_m)$ are given by 
\begin{equation}\label{def:Pkl}
P_{k,l}(\theta_m) := 2 + 2(-1)^m N_{k+l}\left( \theta_m \right),
\end{equation}
and 
\begin{equation}\label{def:Qkl}
Q_{k,l}(\theta_m) := M_l^2 \left( \theta_m \right) + M_l \left(\theta_m \right)  \left\{  N_{l}\left( \theta_m \right) + (-1)^m  N_{k}\left(\theta_m \right) \right\}.
\end{equation}

Our goal of the rest of this section is to show that for all $\theta_m \in [\pi/2, 2\pi/3]$, $G_{k,l}(\theta_m)$ has different signs for consecutive integers $m$'s in $[(k+l)/4, (k+l)/3]$ by proving that a lower bound of $(-1)^m M_{k,l}(\theta_m)$ is greater than the upper bound of $|R_{k,l}(\theta)|$ given in Proposition \ref{prop:maintermplusremainderterm}. 


\subsection{Lower bound of $P_{k,l}(\theta_m)$} Since bounding $(-1)^m M_{k,l}(\theta_m)$ is equivalent to bounding $P_{k,l}(\theta_m)$ and $Q_{k,l}(\theta_m)$, let us first begin by giving a lower bound for $P_{k,l}(\theta_m)$. 

\begin{proposition} \label{prop:LowerboundPkl} For even integers $k > l \geq 10$, and for $\theta_m\in [\pi/2, 2\pi/3]$,  
	$$
	P_{k,l}(\theta_m) \geq  \left\{ \begin{array}{rl}
	1.98222 &\mbox{ if $k+l \equiv 0 \pmod 6,$} \\
1.99970 &\mbox{ if $k+l \equiv 2\pmod 6,$} \\
1.64160 &\mbox{ if $k+l \equiv 4 \pmod 6.$} \\
\end{array} \right.
	$$
\end{proposition}
	
\begin{proof} Let $k > l\geq 10$ be even integers. By (\ref{def:Pkl}),  $P_{k,l}(\theta_m)$ is given by 
	$$
	P_{k,l}(\theta_m) = 2+2(-1)^m N_{k+l}(\theta_m).
	$$
Applying the same argument discussed in (\ref{lemma1}) from the proof of Theorem \ref{theorem:2k3k}, $N_{k+l}(\theta_m)$ is positive and also monotonically increasing as a function of $m  \in [(k+l)/4, (k+l)/3]$. This implies that 
	\begin{equation}\label{ineq:Pkl}
	 P_{k,l}(\theta_m)  
	\geq 2-  2\left( 2\cos \left( \frac{m_{odd}\pi}{k+l} \right) \right)^{-(k+l)}
	\end{equation}
	where $m_{odd}$ denotes the largest odd number in $[(k+l)/4, (k+l)/3]$. Using the notation $k+l = 6q + r$ with $q \in \mathbb{N}$ and $r \in \{ -2, 0 ,2\}$,  a simple calculation reveals that 
	$
	m_{odd} = (k+l)/3- (3+r)/3. 
	$
Inserting this value into (\ref{ineq:Pkl}) to obtain
	$$
	 P_{k,l}(\theta_m) \geq 2 -2\left( 2\cos\left( \frac{\pi}{3} -\left( \frac{3+r}{3} \right) \frac{\pi}{k+l} \right) \right)^{-(k+l)}. 
	$$
By Lemma 2.2 of \cite{reitzes2017zeros} and the identity $\cos(\theta) = \sin(\pi/2 - \theta)$, the right hand side is monotonically increasing as a function in $k+l$. Hence, for $k+l \equiv 0 \pmod 6$, and $k+l \geq 24$, 
	\begin{equation}\label{ineq:Pkl=0}
	P_{k,l}(\theta_m)  \geq 2 -2\Big( 2\cos\left( \frac{\pi}{3} - \frac{\pi}{24} \right) \Big)^{-(24)} \geq 1.98222.
	\end{equation}
By a similar argument, we have that for $k+l \equiv 2 \pmod 6$, and $k+l \geq 26$,
	\begin{equation}\label{ineq:Pkl=2}
	 P_{k,l}(\theta_m)  \geq 2 -2\Big( 2\cos\left( \frac{\pi}{3} - \left(\frac{5}{3}\right)\frac{\pi}{26} \right) \Big)^{-(26)} \geq 1.99970,
	\end{equation}
	and for $k+l \equiv -2 \pmod 6$, and $k+l \geq 22$,
	\begin{equation}\label{ineq:Pkl=-2}
	P_{k,l}(\theta_m) \geq 2 -2\Big( 2\cos\left( \frac{\pi}{3} - \left(\frac{1}{3}\right)\frac{\pi}{22} \right) \Big)^{-(22)} \geq 1.64160.	
	\end{equation}
	By (\ref{ineq:Pkl=0}), (\ref{ineq:Pkl=2}) and (\ref{ineq:Pkl=-2}), we finish the proof.
	\end{proof}


\subsection{Lower bound of $Q_{k,l}(\theta_m)$}
We now turn to bounding $Q_{k,l}(\theta_m)$. 
	\begin{proposition} \label{prop:LowerboundQkl} For even integers $k > l \geq 10$, and for $\theta_m \in [\pi/2, 2\pi/3]$, 
	$$
	Q_{k,l}(\theta_m) \geq  \left\{ \begin{array}{ll}
												-0.31566 &\mbox{ if $l \equiv 0 \pmod 6$,} \\
												-1 &\mbox{ otherwise.} \\
												\end{array} \right.
	$$
\end{proposition}
	
\begin{proof} Let $k > l \geq 10$ be even integers. By (\ref{def:Qkl}), $Q_{k,l}(\theta_m)$ is given by 
 	$$
	Q_{k,l}(\theta_m) = M_l^2 \left( \theta_m \right) + M_l \left(\theta_m \right)  \left\{  N_{l}\left( \theta_m \right) + (-1)^m  N_{k}\left(\theta_m \right) \right\}.
	$$
	Since $N_k(\theta_m) = (2\cos(\theta_m/2))^{-k} < (2\cos(\theta_m/2))^{-l} = N_l(\theta_m)$, the term in curly brackets is always positive. This follows that   $Q_{k,l}(\theta_m) \geq 0$ if $M_l(\theta_m) \geq 0$ and thus we have the desired bound in this case. 
	
	For the rest of the proof, we may assume that $\theta_m \in [\pi/2, 2\pi/3]$ for which $M_l(\theta_m) <0$. First assume  $l\equiv 0 \pmod 6$. By the definition of $Q_{k,l}(\theta_m)$ given above, $ -2 \leq M_k(\theta_m) <0$, $(-1)^m N_k(\theta_m) < N_{l+2}(\theta_m)$, and $N_l(\theta_m)$ is monotonically increasing by (\ref{lemma1}), we obtain 
		\begin{equation}\label{MinQkl=0mod6}
		\begin{aligned}
		Q_{k,l}(\theta_m) &> M_k(\theta_m)  \left\{  N_{l}\left(\theta_m \right) + (-1)^m N_{k}\left(\theta_m \right) \right\}\\
		 &\geq  -2 \left\{  N_{l}\left( \theta^\ast \right) +  N_{l+2}\left(\theta^\ast \right) \right\}
		\end{aligned}
		\end{equation}
where $\theta^\ast$ denotes the largest $\theta$ value in $[\pi/2, 2\pi/3]$ satisfying $M_l(\theta) \leq 0$. By
the aid of Mathematica, we find that 
$$
\theta^\ast = \frac{2\pi}{3} - \frac{\pi}{l}.
$$ 
Inserting this into the right hand side of (\ref{MinQkl=0mod6}), we have that 
	$$
	Q_{k,l}(\theta_m) \geq - 2 \left\{  \left( 2\cos \left( \frac{\pi}{3} - \frac{\pi}{2l}\right)\right)^{-l} +\left( 2\cos \left( \frac{\pi}{3} - \frac{\pi}{2l}\right)\right)^{-(l+2)} \right\}.
	$$
By Lemma 2.2 in \cite{reitzes2017zeros} and the identity $\cos(\theta) = \sin(\pi/2 -\theta)$,  the right hand side is monotonically increasing in $l$. In this case, we find that for $l\equiv 0 \pmod 6$ and $l \geq 12$,
	$$
	Q_{k,l}(\theta_m) \geq - 2\left\{ \left( 2\cos\left( \frac{\pi}{3} - \frac{\pi}{24}\right) \right)^{-12}  + \left( 2\cos\left( \frac{\pi}{3} - \frac{\pi}{24} \right) \right)^{-14}   \right\}= - 0.31566.
	$$
		
		Next, we suppose $l \equiv 2 \pmod 6$ and first consider $\theta_m \in [\pi/2, 2\pi/3 - \pi/3l]$. 
In this case, $Q_{k,l}(\theta_m)$ is bounded above the same way as  the previous case of  $l \equiv 0\pmod 6$. In fact, we find that 
\begin{equation}\label{MinQkl=2mod6}
		Q_{k,l}(\theta_m) \geq -2 \left\{  N_{l}\left( \theta^\ast \right) +  N_{l+2}\left(\theta^\ast \right) \right\}
		\end{equation}
where $\theta^\ast$ now denotes the largest $\theta$ value in $[\pi/2, 2\pi/3 - \pi/3l]$ satisfying $M_l(\theta) \leq 0$. By
the aid of Mathematica,  
$$
\theta^\ast = \frac{2\pi}{3} - \frac{7\pi}{3l}.
$$
	Combining this with (\ref{MinQkl=2mod6}), we find that for $l\equiv 2  \pmod 6$ and $l \geq 14$,
$$
		Q_{k,l}(\theta_m) \geq - 2\left\{ \left( 2\cos\left( \frac{\pi}{3} - \frac{7\pi}{84} \right) \right)^{-14}  + \left( 2\cos\left( \frac{\pi}{3} - \frac{7\pi}{84} \right) \right)^{-16}   \right\}  = - 0.02344. 
$$

We finish the case $l\equiv 2 \pmod 6$ by considering the $\theta_m$ values such that $\theta_m\in(2\pi/3 - \pi/3l,2\pi/3]$. In this case, the negative value of $M_l(\theta_m)$ and the fact that the term in curly brackets of (\ref{def:Qkl}) lies in $(0,2]$ yield
\begin{equation} \label{approxQkl=2}
		Q_{k,l}(\theta_m) \geq M_l^2(\theta_m) + 2 M_l(\theta_m) = 4\cos^2 \left( l \frac{\theta_m}{2} \right) + 4\cos\left( l \frac{\theta_m}{2} \right). 
\end{equation}
Considering the right hand side as a function in $\theta \in (2\pi/3 - \pi/3l,2\pi/3]$, it is straightforward to check that its derivative is negative and thus it is decreasing on that interval and hence takes a minimal value at $\theta = 2\pi/3$. Thus, in this case
	$$
	Q_{k,l}(\theta_m) \geq M_l^2\left( \frac{2\pi}{3} \right) + 2M_l\left( \frac{2\pi}{3} \right) = -1. 
	$$

Finally, assume that $l \equiv 4 \pmod 6$.  We start with considering all $\theta_m$ values that are away from $2\pi/3$, say $\theta_m  \in [ \pi/2, 2\pi/3 - 2\pi/3l]$. Analysis similar to that in the proof of the previous case shows that for $l \equiv 4  \pmod 6$ and $ l \geq 10$, 
$$
\begin{aligned}
Q_{k,l}(\theta_m) 
		&\geq - 2\Big\{ \Big( 2\cos\big( \frac{\pi}{3} - \frac{\pi}{30} \big) \Big)^{-10}  + \Big( 2\cos\big( \frac{\pi}{3} - \frac{\pi}{30} \big) \Big)^{-12}   \Big\}  \\
		&= - 0.68390. 
\end{aligned}
$$
Now suppose that $\theta_m \in (2\pi/3 - 2\pi/3l, 2\pi/3]$.  In this case, using the same reasoning as in (\ref{approxQkl=2}) gives us
$$
Q_{k,l}(\theta_m) \geq M_l^2(\theta_m) + 2 M_l(\theta_m) \geq M_l^2\left( \frac{2\pi}{3} \right) + 2M_l\left( \frac{2\pi}{3} \right) = -1. 
$$
We finish the proof of Proposition \ref{prop:LowerboundQkl}.
\end{proof}


\subsection{Proof of Theorem \ref{theorem:k+l}}
\begin{proof} By Proposition \ref{prop:maintermplusremainderterm}, the function $G_{k,l}(\theta) = e^{i(k+l)\theta/2}(E_kE_l + E_{k+l})(e^{i\theta})$ can be written as
$$
G_{k,l}(\theta) = M_{k,l}(\theta) + R_{k,l}(\theta)
$$
where Propositions \ref{prop:LowerboundPkl} and \ref{prop:LowerboundQkl} showed for even integers $k > l \geq 10$, and for $\theta_m = 2m\pi/(k+l) \in [\pi/2, 2\pi/3]$,  
	$$
	(-1)^mM_{k,l}(\theta_m) \geq  \left\{ \begin{array}{rl}
												1.32594 &\mbox{ if $l \equiv 0 \pmod 6,$} \\
												0.64161 &\mbox{ otherwise.} \\
												\end{array} \right.
	$$

Comparing this with the upperbound of $|R_{k,l}(\theta)|$ given in Proposition \ref{prop:maintermplusremainderterm}, we find that $$(-1)^mM_{k,l}(\theta_m) > |R_{k,l}|$$ and hence the function 
 $G_{k,l}(\theta_m)$ has different signs for consecutive $m$'s integers so that $\theta_m \in [\pi/2, 2\pi/3]$. 
 
 Thus, the intermediate value theorem guarantees that the number of zeros of $G_{k,l}(\theta)$ and hence $\left(E_kE_l + E_{k+l}\right)(e^{i\theta})$ in $(\pi/2, 2\pi/3)$ is at least the number of $\theta_m$ in $[\pi/2, 2\pi/3]$ minus 1. With the notation $k + l = 12n + s$ and $s \in \{ 0,4,6,8,10,14\}$, considering all 6 cases,  a straightforward counting argument shows that there are $n+1$ of $\theta_m$ values in that interval. 
	
	Since $E_k(\tau)E_l(\tau)+E_{k+l}(\tau)$ has at most $n$ zeros by the valence formula (\ref{valenceformula}), and the above argument shows that there are at least $n$ zeros on the arc $\mathcal{A}$,  we have located all zeros of $E_k(\tau)E_l(\tau)+E_{k+l}(\tau)$ in the fundamental domain lie on the arc $\mathcal{A} = \{ e^{i\theta} : \pi/2\leq \theta \leq 2\pi/3\}$. 
\end{proof}

\section{Acknowledgment}
The author wishes to express his gratitude to Professor Holly Swisher for her support and valuable feedback. This paper is part of the author's Ph.D. dissertation, written under the supervision of Professor Holly Swisher at Oregon State University.  

\section*{References}

\bibliography{zeroskl}

\end{document}